\documentclass[11pt]{article}

\usepackage{amsfonts,amsmath,amsthm,latexsym,color,epsfig,mathrsfs}
\newcommand\numberthis{\addtocounter{equation}{1}\tag{\theequation}}

\newtheorem{thm}{Theorem}

\newtheorem{lem}[thm]{Lemma}

\newtheorem{prop}[thm]{Proposition}

\title{Large unavoidable subtournaments}
\author{Eoin Long \thanks{School of Mathematical Sciences, Tel Aviv University, Tel Aviv, Israel. E-mail: eoinlong@post.tau.ac.il}}
\date{18 February 2015}

\begin{document}

\maketitle

\begin{abstract}
	Let $D_k$ denote the tournament on $3k$ 
	vertices consisting of three disjoint vertex classes
	$V_1$, $V_2$ and $V_3$ of size $k$, each 
	oriented as a transitive subtournament, and with edges 
	directed from $V_1$ to $V_{2}$, from $V_2$ to $V_{3}$ 
	and from $V_3$ to $V_{1}$. Fox 
	and Sudakov proved that given a natural number $k$ and 
	$\epsilon > 0$ there is $n_0(k,\epsilon )$ such 
	that every tournament of order 
	$n \geq n_0(k,\epsilon )$ which is 
	$\epsilon $-far from being transitive 
	contains $D_k$ as a subtournament. 
	Their proof showed that $n_0(k,\epsilon ) \leq 
	\epsilon ^{-O(k/\epsilon ^2)}$ and they
	conjectured that this  could be reduced to 
	$n_0(k, \epsilon ) \leq \epsilon ^{-O(k)}$. 
	Here we prove this conjecture.
\end{abstract}

\section{Introduction}
	Ramsey theory refers to a large and active branch of combinatorics 
	mainly concerned with understanding which properties of a structure 
	are preserved in dense substructures or upon finite partition. 
	It is often introduced with the phrase `complete disorder is 
	impossible', attributed to Motzkin, and part of the subject's 
	growth can be attributed to the surprising variety of contexts in 
	which this philosophy can be applied (for a small sample, see 
	\cite{PaAg}, \cite{AlSp}, \cite{Furs}, \cite{GRS}).

	A central result in the area is Ramsey's theorem 
	\cite{Ramsey}, which says that given any natural number 
	$k$, there is an integer $N$ such that every 2-colouring of 
	the edges of the complete graph $K_N$ contains a monochromatic copy of $K_k$. An 
	important problem in the area is to estimate the smallest value 
	of $N$ for which the theorem holds, denoted $R(k)$. It is known 
	that $2^{(1/2 +o(1))k} \leq R(k) \leq 4^{(1+o(1))k}$ (see 
	\cite{Erd}, \cite{Sp}, \cite{ErdSz}, \cite{C}). 
	
	Clearly it is not possible to guarantee 
	the existence of non-monochromatic cliques in general 2-colourings 
	of $K_N$. 
	Bollob\'as raised the 
	question of which 2-coloured subgraphs occur in  
	2-colourings of $K_N$ where each colour appears on 
	at least $\epsilon $ proportion of the edges. Let ${\mathcal F}_k$ 
	denote the collection of 2-coloured graphs of order $2k$, 
	in which one colour appears as either a clique of order 
	$k$ or two disjoint cliques of order $k$. Bollob\'as 
	asked whether, given a natural $k$ and $\epsilon > 0$ 
	there is $M = M(k, \epsilon )$ with the following property: in 
	every 
	2-colouring of the edges of $K_{M}$ containing both 
	colours on at least $\epsilon $ proportion of the edges,
	some element of ${\mathcal F}_k$ appears as a coloured subgraph. 
	Cutler and Mont\'agh \cite{CM} answered this question 
	in the affirmative and proved that it is possible to take 
	$M(k,\epsilon ) \leq 4^{k/\epsilon }$. Fox and Sudakov 
	\cite{FS} subsequently improved this bound to show that $M(k, \epsilon ) \leq \epsilon ^{-ck}$, 
	for some constant $c>0$. 
	As shown in \cite{FS}, this bound is tight up 
	to the value of the constant $c$ in the exponent, which can be seen 
	by taking a random 2-colouring of a graph on 
	$\epsilon ^{-(k-1)/2}$ vertices with appropriate densities.
	
	Here we will be concerned with an analogous question for tournaments.
	A \emph{tournament} is a directed graph obtained 
	by assigning a direction to the edges of a complete 
	graph. A tournament is said to be \emph{transitive} 
	if it is possible to order the vertices of the 
	tournament so that all of its edges point in the 
	same direction. Let $T(k)$ denote the smallest 
	integer such that every tournament on $T(k)$ vertices  
	contains a transitive subtournament on $k$ vertices. 
	A classic result due to Erd\H{o}s and Moser \cite{EM} shows 
	that $T(k)$ is finite for all $k$ and gives 
	that $2^{(k-1)/2} \leq T(k) \leq 2^{k-1}$ (in fact, as pointed out 
	by the referee, the upper bound here is 
	attributed to Stearns in \cite{EM}).
	
	As in the case of $2$-colouring graphs, it is natural
	to ask which subtournaments must occur in every
	large tournament which is `not too similar' to a 
	transitive tournament. An $n$-vertex tournament 
	$T$ is \emph{$\epsilon$-far from 
	being transitive} if in any ordering of the vertices 
	of $T$, the direction of at least $\epsilon n^2$ 
	edges of $T$ must be switched in order to obtain a 
	transitive tournament. In \cite{FS}, Fox and Sudakov 
	asked the following question: given 
	$\epsilon>0$, which subtournaments must an $n$-vertex 
	tournament which is $\epsilon $-far from being transitive 
	contain? 
	
	For any natural number $k$, let $D_k$ denote the
	tournament on $3k$ vertices consisting of three 
	disjoint vertex classes $V_1$, $V_2$ and $V_3$ of size 
	$k$, each 
	oriented as a transitive subtournament, 
	and with all edges directed from $V_1$ to $V_{2}$, 
	from $V_2$ to $V_{3}$ and from $V_3$ to $V_{1}$.
	Taking $T = D_{n/3}$ we obtain an $n$-vertex 
	tournament which is $\frac{1}{9}$-far from being transitive 
	and whose only subtournaments are contained in $D_k$ 
	for some $k$. Thus, subtournaments of $D_k$ are the 
	only candidates for unavoidable tournaments which 
	occur in large tournaments that are $\epsilon $-far from 
	transitive for small $\epsilon $. In \cite{FS}, Fox and 
	Sudakov proved that subtournaments necessarily appear in 
	large tournaments which are $\epsilon $-far from being 
	transitive.
	
	\begin{thm}[Fox--Sudakov]
	\label{F_k subtournament}
	Given $\epsilon > 0$ and a natural number $k$, there is 
	$n_0(k,\epsilon )$ such that if $T$ is a tournament 
	on $n \geq n_0(k,\epsilon )$ vertices which is 
	$\epsilon $-far from being transitive, then $T$ 
	contains $D_k$ as a subtournament. Furthermore $n_0(k, \epsilon)
	\leq \epsilon ^{-ck/\epsilon ^2}$, for some absolute 
	constant $c>0$. \endproof
	\end{thm} 
	
	The authors in \cite{FS} conjectured that this bound can 
	be further reduced to $n_0(k,\epsilon ) \leq 
	\epsilon ^{-Ck}$ for some absolute constant $C>0$. 
	This order of growth agrees with high probability 
	with a random tournament obtained by directing 
	edges backwards independently with probability 
	$\approx \epsilon $. Here we prove this conjecture.
	
\begin{thm}
	\label{improved D_k bound}
	There is a constant $C>0$ such that for  
	$\epsilon > 0$ and any natural number $k$ we have 
	$n_0(k,\epsilon ) \leq \epsilon ^{-Ck}$.
\end{thm}

	Before beginning on the proof let us mention two other 
	results related to Theorems \ref{F_k subtournament} and 
	\ref{improved D_k bound}. 
	A tournament $T$ is said to be \emph{$c$-colourable} if it is possible to 
	partition $V(T)$ into $c$ subsets, each of which is a transitive 
	subtournament. The \emph{chromatic number} $\chi (T)$ of a tournament $T$ 
	equals the smallest value of $c$ such that $T$ is $c$-colourable. 
	A tournament $H$ is said to be a \emph{hero} if every $H$-free 
	tournament has bounded chromatic number. The definition of a hero 
	was introduced in 
	by Berger, Choromanski, Chudnovksy, Fox, Loebl, Scott, Seymour and 
	Thomass\'e in \cite{BCCFLSST} and their main result gave an 
	explicit description of heroes. This notion was recently extended by 
	Shapira and Yuster \cite{SY}. A tournament $H$ is said to be 
	\emph{$c$-unavoidable} if for every $\epsilon >0$ and 
	$n \geq n_0(\epsilon ,H)$, every $n$-vertex tournament $T$ that is 	
	$\epsilon$-far 
	from satisfying $\chi (T) \leq c$ contains a copy of $H$. 
	A tournament $H$ is said 
	to be \emph{unavoidable} if it is $c_H$-unavoidable for some constant 
	$c_H$. Clearly a tournament is $1$-chromatic if and only if it is 
	transitive. Thus from Theorem 
	\ref{F_k subtournament} and the discussion preceding it, 
	$1$-unavoidable tournaments are precisely those tournaments which 
	appear as subtournaments of $D_k$ for some $k$. In \cite{SY} this 
	result was extended to show that a tournament $H$ is unavoidable 
	iff it is a transitive blowup of a hero (see \cite{BCCFLSST} and 
	\cite{SY} for the precise definitions).

	\textit{Notation:} Given a tournament $T$, we write $V(T)$ to denote 
	its vertex set and $E(T)$ to denote the directed edge set of $T$. 
	Given $v\in V(T)$ and a set $S \subset V(T)$, let $d_S^-(v) 
	:= |\{u\in S:\overrightarrow {uv} \in E(T)\}|$ and $d_S^+(v) := 
	|\{u\in S:\overrightarrow {vu} \in E(T)\}|$. We will also write $T[S]$ 
	to denote the induced subtournament of $T$ on vertex set $S$. 
	Given $B \subset E(T)$, we write 
	$d^-_B(v) = |\{u \in V(T): \overrightarrow{uv} \in B\}|$ and 
	$d^+_B(v) = |\{u \in V(T): \overrightarrow{vu} \in B\}|$. For 
	an ordering $v_1,\ldots ,v_{|T|}$ of $V(T)$ and $1\leq i<j\leq |T|$, let 
	$[v_i,v_j] := \{v_i,v_{i+1},\ldots ,v_j\}$. Lastly, all $\log $ functions will 
	be to the base $2$.
	
	\section{Finding many long backwards edges in $T$}
	
	In \cite{FS}, Theorem \ref{F_k subtournament} 
	was deduced from two results of 
	independent interest. The first result  showed 
	that any tournament which is $\epsilon$-far from 
	being transitive must contain many directed 
	triangles.
	
	\begin{thm}[Theorem 1.3 in \cite{FS}]
	\label{number of directed triangles}
	Any $n$-vertex tournament $T$ which is 
	$\epsilon $-far from being transitive contains 
	at least $c \epsilon ^2 n^3$ directed triangles, 
	where $c>0$ is an absolute constant. \endproof
	\end{thm}
	
	As pointed out in \cite{FS}, this bound is also best possible 
	in general, as can be seen from the following tournament. 
	Let $T$ be given by taking $k$ copies of 
	${D}_{n/3k}$, say on disjoint vertex sets 
	$V_{1},\ldots ,V_k$ with all edges between 
	$V_i$ and $V_j$ directed forward, for $i<j$. As at least 
	$(n/3k)^2$ edges from each copy of $D_{n/3k}$ must be 
	reoriented in order to obtain a transitive tournament, 
	$T$ is $k(1/3k)^2 = 1/9k$ far 
	from being transitive, but contains only 
	$k.(n/3k)^3 = n^3/27k^2$ directed triangles. Taking 
	$\epsilon = 1/9k$, we see that the growth rate here 
	agrees with that given by Theorem 
	\ref{number of directed triangles} up to constants.
	
	Our first improvement in the bound for $n_0(k,\epsilon )$ 
	comes from showing that any tournament which is $\epsilon $-far 
	from being transitive must either contain many more 
	directed triangles than the number given 
	in Theorem \ref{number of directed triangles} or 
	contain a slightly smaller subtournament which 
	is $2\epsilon $-far from being transitive. This density 
	increment argument will allow one of the factors 
	of $\epsilon $ to be removed from the exponent in 
	the bound on $n_0(k,\epsilon )$ in Theorem \ref{F_k subtournament}.

	Given an ordering $v_1,\ldots ,v_{|T|}$ of the vertices of a 
	tournament $T$, edges of the form $\overleftarrow{v_iv_j}$ with 
	$i<j$ are called \emph{backwards edges}. We will often list the 
	vertices of tournaments in an order which minimizes the number 
	of backwards edges. Such orderings are said to be \emph{optimal}. 
	The following proposition gives some simple but useful properties 
	of optimal orderings.
		
	\begin{prop}
		\label{simple properties of ordrering}
		Suppose that $T$ is a tournament on $n$ vertices and let 
		$v_1,\ldots ,v_n$ be an optimal ordering of $V(T)$. Then 
		the following hold:		
		\begin{enumerate} 
			\item For every $i,j \in [n]$ with $i<j$ we have 
				\begin{itemize}
					\item $d^+_{[v_{i+1},v_j]}(v_i) \geq 
						(j-i)/2$;
					\item $d^-_{[v_{i},v_{j-1}]}(v_j) \geq 
						(j-i)/2$.
				\end{itemize}
			\item If $T[v_{i+1},v_j]: = T[\{v_{i+1},\ldots ,v_j\}]$ has 
			$\delta (j-i)^2$ backwards edges in this ordering then 
			the subtournament $T[v_{i+1},v_{j}]$ is $\delta $-far 
			from being transitive.
		\end{enumerate}
	\end{prop}
	
	\begin{proof} If we had $d^+_{[v_{i+1},v_j]}(v_i) <(j-i)/2$, then the ordering $v_1,\ldots, v_{i-1},v_{i+1},\ldots ,v_j,v_i, v_{j+1},\ldots ,v_n$
	would decrease the number of backwards edges of $T$. A similar switch works 
	if $d^-_{[v_{i},v_{j-1}]}(v_j) < (j-i)/2$. Lastly, if $v_{k_1} ,\ldots , 
	v_{k_{j-i}}$ was an ordering of $[v_{i+1},v_j]$ with fewer than 
	$\delta (j-i)^2$ backwards edges, then the order of $V(T)$ given by 
	$v_1,\ldots ,v_{i},v_{k_1} ,\ldots , v_{k_{j-i}}, v_{j+1}, \ldots , v_n$ 
	would have less backwards edges than $v_1,\ldots ,v_n$. \end{proof}
	
	Given an ordering $v_1,\ldots ,v_n$ of $V(T)$ with a backwards edge 
	$\overleftarrow{v_iv_j}$ ($i<j$), the edge 
	$\overleftarrow{v_i	v_j} \in B$ is said to have \emph{length} $j-i$.
	
	\begin{lem}
	\label{many long edges or density boost}
	Suppose that $T$ is a tournament on $n$ 
	vertices which is $\epsilon $-far from being 
	transitive and let $v_1,\ldots ,v_n$ be an optimal 
	ordering of $V(T)$. Let $B$ denote the collection  
	of backwards edges in this ordering. Then one of the 
	following holds:
	\begin{enumerate}
		\item The subset $B'$ of $B$ consisting 
		of those edges of length at 
		least $n/16$ satisfies $|B'| \geq |B|/4$;
		\item $T$ contains a subtournament on 
		at least $n/8$ vertices which is 
		$2\epsilon $-far from being transitive.
	\end{enumerate}
	\end{lem}
	
	\begin{proof}
	We can assume that $T$ itself is not $2\epsilon$-far 
	from being transitive, as otherwise 2. above  would trivially 
	hold. Thus $\epsilon n^2 \leq |B| < 2\epsilon n^2$. Let us 
	assume that $|B'| < |B|/4$, i.e. 1. fails. Note that this gives 
	$n\geq 16$. We wish to 
	show that there exists  $S \subset V(T)$ with $|S| \geq n/8$ such 
	that $T[S]$ is $2\epsilon $-far from being transitive. To prove 
	this, by part 2 of Proposition \ref{simple properties of ordrering}, 
	it suffices to find an interval $v_{i+1},\ldots ,v_j$ with 
	$j-i \geq n/8$ containing at least $2\epsilon(j-i)^2$ 
	edges from $B$.	

	If either $T_{first} = T[v_1,\ldots ,v_{n/8}]$ or $T_{last} = T[v_{7n/8 + 1},\ldots ,v_n]$ 
	have at least $2\epsilon (n/8)^2 = \epsilon n^2/32$ backwards edges then we done. Otherwise, let 
	$E$ denote the subset of $B$ consisting of those backwards edges not 
	in $B'$ and not in $T_{first}$ or $T_{last}$. From the above bounds 
	\begin{align}
		\label{equation: bound on |E|}
		|E| > |B| - |B'| - 2\frac{\epsilon n^2}{32} 
		> \frac{3|B|}{4} - \frac{\epsilon n^2}{16} \geq \frac{\epsilon n^2}{2}.
	\end{align}
		
	Now given $i\in [0,7n/8]$, let $T_i$ denote the subtournament of $T$ 
	which is given by 
	$T_i = T[\{v_{i+1},\ldots ,v_{i + n/8}\}]$. Choose $i\in [0,7n/8]$ uniformly 
	at random and let $E_i$ denote the random variable which counts the number of 
	edges of $E$ which lie in $T_i$. As each element $e \in E$ has 
	length at most $n/16$, with at least one endpoint in $\{v_{n/8},\ldots ,v_{7n/8}\}$, 
	there are at least $n/16$ choices of $i$ with $e \in T_i$. As $n\geq 8$, 
	this gives
	\begin{align*}
		\mathbb {P}(e \in T_i) \geq \frac {n/16}{7n/8+1} \geq \frac{1}{16}.
	\end{align*}
	By linearity of expectation, combined with \eqref{equation: bound on |E|} 
	this gives 
	\begin{align*}
	\mathbb {E}(E_i) 
		& = 	
	\sum _{e \in E} \mathbb {P}(e \in T_i) \geq 
			\frac{|E|}{16} \geq 
			\frac{\epsilon n^2}{32} = 
			2\epsilon (\frac{n}{8})^2.
	\end{align*}
	Fix a value of $i$ such that $E_i$ is at least as large as its expectation. Then  as
	$T_i$ has $n/8$ vertices and at least 
	$2\epsilon (n/8)^2$ backwards edges. By Proposition 
	\ref{simple properties of ordrering}, $T_i$ is 
	$2 \epsilon$-far from being transitive, as required. 
	\end{proof}
	
	\section{Finding many directed triangles in $T$}	
	
	Our second lemma will show that in a tournament  
	with few backwards edges, many of which have large 
	length, there is a large subset of backwards edges 
	which all lie in many directed triangles.
	
	\begin{lem}
	\label{many disjoint triangles}
	Let $T$ be an $n$-vertex tournament with an optimal 
	ordering $v_1,\ldots ,v_n$ 
	and let $B$ denote the set of backwards edges 
	in this ordering, $|B| = \alpha n^2$.
	Suppose that the subset $B' \subset B$ of 
	backwards edges with length at least $n/16$ 
	satisfies $|B'| \geq \alpha n^2/4$. 
	Then, provided that $\alpha \leq 2^{-16}$, there 
	exists $B''\subset B'$ satisfying 
	$|B''| \geq |B'|/2$ with the property that 
	each edge of $B''$ lies in at least $n/64$ 
	directed triangles in $T$.
	\end{lem}
	
	\begin{proof}
	Given $B'$ as in the statement of the lemma, 
	let $B'' \subset B'$ be the set
	\begin{align*}
		B'' := \{\overleftarrow{v_iv_j} \in B': 
				\mbox{ either } 
				d^-_{[v_{i+1},v_j]}(v_i) \leq 
				4\alpha ^{1/2} n
				\mbox{ or } 
				d^+_{[v_{i},v_{j-1}]}(v_j) \leq 
				4\alpha ^{1/2} n \}
	\end{align*}
	We first claim that $|B''| \geq |B'|/2$. To see this 
	let $S_- = \{v_i \in V(T): d_B^-(v_i) \geq  
	4 \alpha ^{1/2} n\}$ and let $S_+ = \{v_i \in V(T): d_B^+(v_i) \geq 4 \alpha ^{1/2}n \}$. 
	Using that 
	\begin{align*}
		4\alpha ^{1/2} n|S_-| \leq \sum _{i \in S_-} d_B^-(v_i) \leq 
	\sum _{i \in [n]} d_B^-(v_i) = |B|,
	\end{align*} 
	gives $|S_-| \leq |B|/4 \alpha ^{1/2} n = \alpha ^{1/2}n/4$. Similarly 
	$|S_+| \leq  \alpha ^{1/2}n/4$. But all edges 
	$\overleftarrow{v_iv_j} \in B'\setminus B''$ 
	have $v_i \in S_-$ and $v_j \in S_+$. This 
	gives 
	\begin{align*}
		|B' \setminus B''| \leq |S_-||S_+| 
		 \leq 
		 ( \alpha ^{1/2}n/4 )^2 
		 = \alpha n^2/16.
	\end{align*}
	But then $|B''| \geq |B'| - \alpha n^2/16 \geq |B'|/2$, as claimed.
	
	Now recall that by the definition of $B'$, for every 
	$\overleftarrow{v_iv_j} \in B''$ we have 
	$j-i \geq n/16$. Also, by Proposition \ref{simple properties of ordrering} 
	part 1 we have $d^+_{[v_{i+1},v_j]}(v_i) \geq (j-i)/2 $ and 
	$d^-_{[v_i,v_{j-1}]}(v_j) \geq (j-i)/2 $.
	Furthermore, as $\overleftarrow{v_iv_j} \in B''$ we must 
	have either
	\begin{align*}
		d^+_{[v_{i+1},v_j]}(v_i) \geq (j-i) - 4 \alpha ^{1/2}n \geq 3(j-i)/4
	\end{align*}
	or
	\begin{align*}
	d^-_{[v_i,v_{j-1}]}(v_j) \geq (j-i) - 4 \alpha ^{1/2}n 
		\geq 3(j-i)/4.
	\end{align*}
	The inequalities hold here since $\alpha \leq 2^{-16}$ and $j-i \geq n/16$
	gives $(j-i)/4 \geq n/2^6 \geq 4\alpha ^{1/2}n$. 
	Thus for every edge $\overleftarrow{v_iv_j} \in B''$ there 
	are at least $(j-i)/4 \geq n/2^6$ vertices $v_k \in \{v_{i+1},\ldots ,v_{j-1}
	\}$ such that $\overrightarrow{v_iv_k}$ and $\overrightarrow{v_kv_j}$ are edges. 
	But this gives that every edge of $B''$ lies in at least 
	$n/2^6$ directed triangles, as claimed.	
	\end{proof}
	
	\section{Finding a copy of $D_k$ in $T$}
	
	The second half of our argument is based on another result 
	from \cite{FS}. Here the authors proved that the following holds:
	
	\begin{thm}[Theorem 3.5 in \cite{FS}]
	\label{many directed triangles give a copy of D_k}
	Any $n$-vertex tournament with at least $\delta n^3$ directed 
	triangles contains $D_k$ as a subtournament provided that 
	$n \geq \delta ^{-4k/ \delta }$. \endproof
	\end{thm}
	
	By combining Lemma \ref{many long edges or density boost} and 
	Lemma \ref{many disjoint triangles} with Theorem 
	\ref{many directed triangles give a copy of D_k} it is 
	already possible to improve the bound $n_0(k,\epsilon )$, to 
	show that $n_0(k, \epsilon ) \leq \epsilon ^{-ck/\epsilon }$ for some 
	fixed constant $c>0$. To remove the additional $\epsilon $ term from 
	the exponent, we need to modify Theorem 
	\ref{many directed triangles give a copy of D_k}. 

	The next lemma shows that if many directed triangles in Theorem 
	\ref{many directed triangles give a copy of D_k} occur in a very 
	unbalanced manner, meaning that each of these triangles contain an edge from 
	a small set, the lower bound on $n$ in Theorem \ref{many directed triangles give a copy of D_k} can 
	be reduced. Note that this is exactly the situation given by Lemma 
	\ref{many disjoint triangles}.
	
	\begin{lem}
		\label{many imbalanced triangles}
		Let $T$ be an $n$-vertex tournament and let $E$ be a set of $\beta n^2$ edges in $T$. Suppose that each edge 
		of $E$ occurs in at least $\gamma n$ directed triangles in $T$. 
		Then $T$ contains $D_k$ as a subtournament provided  
		$n \geq \beta ^{-100k/ \gamma }$.
	\end{lem}	
	
	The proof modifies the proof of Theorem 
	\ref{many directed triangles give a copy of D_k} in \cite{FS}, but 
	as the details are somewhat technical, we have included the proof 
	in full. We will use the following formulation of the dependent 
	random choice method (see \cite{DRC}).
	
	\begin{lem}
		\label{DRC}
	Let $G = (A,B,E)$ be a bipartite graph with 
	$|A| = |B| = n$ and $\alpha n^2$ edges. 
	Given $d, l \in {\mathbb N}$, there exists a 
	set 
	$A' \subset A$ with $|A'| \geq \alpha ^l n -1$ 
	such that every $d$-set in $A'$ has at least 
	$n^{1-d/l}$ common neighbours in $B$. \endproof
	\end{lem}
	
	We will also use of the following bound for the Zarankiewicz 
	problem, due to K\"ovari, S\'os and Tur\'an (see \cite{Zarank}, 
	\cite{KST}). Here it was shown that 
	any bipartite graph $G = (A,B,E)$, with $|A| = m$, $|B|=n$, 
	which does not contain $K_{s,t}$ as a subgraph, with $s$ 
	vertices in $A$ and $t$ in $B$ satisfies 
	\begin{align}
		\label{Zarankiewicz bound}
		e(G) \leq (s-1)^{1/t}(n-t+1)m^{1-1/t} + (t-1)m.
	\end{align}		
	
	\begin{proof}[Proof of Lemma \ref{many imbalanced triangles}]
	To begin, pick a random equipartition of $V(T)$ into three 
	sets $V_1, V_2$ and $V_3$, each with size $n/3$. For each 
	edge $e\in E$, let $Q^{(i)}_e$ denote the number of 
	vertices $v \in V_i$ which form a directed triangle with $e$ in $T$. 
	Let $E_{good}$ denote the collection of (random) edges 
	$e =\overrightarrow{xy} \in E$	with $x \in V_1$ to $y\in V_2$ and 
	$Q^{(3)}_e \geq \gamma n/3$. For all $e \in E$, we have
	\begin{align*}
		 {\mathbb P}(e \in E_{good}) =
		 {\mathbb P}(e \in \overrightarrow{V_1V_2} \mbox{ and } 
		Q^{(3)}_e \geq \gamma n/3) &= 
		{\mathbb P}(Q^{(3)}_e \geq \gamma n/3| e \in \overrightarrow{V_1V_2})
		\times 
		{\mathbb P}(e \in \overrightarrow{V_1V_2})\\
		& \geq  \frac{1}{3}\times \frac{|V_1||V_2|}{n(n-1)} \geq \frac{1}{27}.
		\numberthis \label{prob of edge being good}
	\end{align*}
	To see the inequality here, note that as
	$|V_3| \geq |V_1\setminus \{x\}|, |V_2 \setminus \{y\}|$ 
	we have ${\mathbb P}(Q^{(3)}_e \geq \gamma n/3| e\in \overrightarrow{V_1V_2}) 
	\geq {\mathbb P}(Q^{(i)}_e \geq \gamma n/3| e\in \overrightarrow{V_1V_2})$ 
	for $i\in \{1,2\}$. As $e\in E$ we also have  
	$\sum _{i=1}^3 Q_e^{(i)} \geq 
	\gamma n$ and so 
	\begin{align*}
	3{\mathbb P}(Q^{(3)}_e \geq \gamma n/3| e\in \overrightarrow{V_1V_2}) 
		& \geq 
	\sum _{i=1}^3 {\mathbb P}(Q^{(i)}_e
		\geq \gamma n/3| e\in 
	\overrightarrow{V_1V_2}) \\
		&\geq 
	{\mathbb P}(Q^{(i)}_e \geq \gamma n/3 \mbox { for some }i| e
	\in\overrightarrow{V_1V_2}) = 1.
	\end{align*}
	By \eqref{prob of edge being good} we have 
	${\mathbb E}(|E_{good}|) \geq |E|/27 \geq \beta n^2/27$. 
	Fix a partition with $|E_{good}|$ at least this big. 
	
	Now take $H$ to denote the bipartite graph between sets $V_1$ 
	and $V_2$ whose edge set is $E_{good}$. From the previous 
	paragraph $|e(H)| \geq 
	\beta n^2/27 = \frac{\beta}{3} (\frac{n}{3})^2$. Applying Lemma 
	\ref{DRC} to $H$ with $d = 3k/\gamma$ and $l=4d$ we can find a set 
	in $W_1 \subset V_1$ with $|W_1| \geq (\beta /3 )^{l}|V_1| -1 \geq n^{1/2}$ 
	such that every $d$-set in $W_1$ has at least $(n/3)^{1-d/l} = 
	(n/3)^{3/4} \geq n^{1/2}$ 
	common neighbours in $V_2$. The inequality on 
	$|W_1|$ here holds since 
	\begin{align*}
 	(\beta /3)^{l}|V_1| -1 \geq {\beta ^{3l}}\frac{n}{3} - 1 
 	\geq 2\beta ^{4l}n - 1 \geq 2n^{1/2} - 1 
 	\geq n^{1/2}. 
	\end{align*}
	using that $1/3 \geq \beta ^2$ and $\beta ^{l} \leq 1/6$ (since 
	$\beta \leq 1/2$, $l\geq 4$) and $n \geq \beta ^{-100k/\gamma } \geq 
	\beta ^{-8l}$.
	
	Now by applying the Erd\H{o}s--Moser theorem to $W_1$, we find 
	a transitive subtournament $T_1$ on vertex set $S_1 \subset W_1$ 
	with $|S_1| \geq \log |W_1| \geq \log n^{1/2} \geq d$. Letting 
	$N_H[S_1] \subset V_2$ denote the common neighbourhood of $S_1$ 
	in $H$, by choice of $W_1$ we 
	have $|N_H[S_1]| \geq n^{1/2}$. Again apply the Erd\H{o}s--Moser 
	theorem to $N_H[S_1]$, we find $S_2 \subset N_H(S_1)$ with 
	$|S_2| \geq \log |N_H[S_1]| = \log n^{1/2} \geq d$ vertices. By 
	the construction of $H$, this gives that all edges of $T$ 
	between $S_1$ and $S_2$ are directed from $S_1$ to $S_2$.
		
	For the next section of the argument, fix a matching of size 
	$d$ within this bipartite directed subgraph $T[S_1,S_2]$, 
	say with edges 	$\{e_1,\ldots ,e_d\}$. 
	As each edge $e_i \in E_{good}$, we have
	$Q^{(3)}_{e_i} \geq \gamma n/3$ for all $i\in [d]$.
	Now consider the bipartite graph $G$ on vertex set 
	$A = \{e_1,\ldots ,e_d\}$ and $V_3$ in which 
	$e_i \in A$ is joined to $v\in V_3$ if together 
	the vertices of $e_i$ and $v$ form a 
	directed triangle in $T$. As $Q_{e_i} \geq \gamma n/3$ for 
	all $i\in [d]$, we see $e(G) \geq d\gamma n/3 = kn$. 
	
	We now claim that in $G$ there exists $A' \subset A$ and 
	$V_3' \subset V_3$ with $|A'| \geq k$ and $|V_3'| \geq 
	n^{1/2}$ such that $G[A',V_3']$ is complete. Indeed, by 
	\eqref{Zarankiewicz bound}, 
	if $G$ does not contain a complete bipartite 
	subgraph $G'$ with $k$ vertices in $A$ and 
	$n^{1/2}$ vertices in $V_3$, then the 
	number of edges in $G$ satisfies	
	\begin{align*}
	e(G) 
		& < 
	(n^{1/2} - 1)^{1/k}(d - k + 1)(n/3)^{1-1/k} 
	+ 
	(k - 1)n/3\\ 
		& < 
	(dn^{- 1/2k} + k/3)n 
		 \leq 
	5kn/6 < e(G).
	\end{align*} 
	To see the second last inequality, note that $n^{1/2k} \geq  
	\beta ^{-12/\gamma } \geq 2^{12/\gamma } \geq e^{6/\gamma }
	\geq 6/\gamma $, as $\beta \leq 1/2$ and $e^x \geq x$ for all $x$.
	This gives $dn^{-1/2k} \leq d\gamma /6 = k/2$. This contradiction 	
	shows that must exist some set of $k$ edges $\{e_{i_1},\ldots ,e_{i_k}\} 
	\subset A$ which is completely joined to a set $W_3\subset V_3$ of 
	size at least $n^{1/2}$. To complete the proof of the 
	lemma, apply the Erd\H{o}s--Moser theorem a final time to $W_3$ to find 
	a transitive subtournament of size $\log n^{1/2} > d > k$ 
	on vertex set $S_3$. For $i=1,2$, let $U_i$ denote the sets $U_i \subset V_i$ 
	which occur in the edges $\{e_{i_1},\ldots ,e_{i_k}\}$. Also let $U_3 \subset 
	S_3$ be a set with $|U_3| = k$. 
	
	We claim that $T[U_1\cup U_2 \cup U_3]$ forms a subtournament isomorphic 
	to $D_k$. Indeed, $|U_i| = k$ for all $i\in [3]$ and $T[U_i]$ is transitive 
	since $U_i \subset S_i$. Also, all edges in $T$ between $U_1$ and $U_2$ are 
	directed from $U_1$ to $U_2$, since $U_i \subset S_i$. Lastly, from 
	the definition of $H$, each $u \in U_3$ forms a directed triangle in $T$ with 
	$e_{i_j}\in \overrightarrow{U_1U_2}$ for all $j \in [k]$ giving that 
	all edges of $T$ are directed from $U_2$ to $U_3$ and from $U_3$ to $U_1$.
	\end{proof}
	
	We can now complete the proof of Theorem \ref{improved D_k bound}.	
	
\begin{proof}[Proof of Theorem \ref{improved D_k bound}]
	Take $c\geq 1$ to be a constant such that Theorem \ref{F_k subtournament} holds 
	and set $C = 2^{33}c$. We will show that an $n$-vertex tournament $T$ 
	which is $\epsilon $-far from being transitive contains $D_k$ as 
	a subtournament, provided $n\geq \epsilon ^{-Ck}$. 
	
	To begin,
	choose $i\in {\mathbb N} \cup \{0\}$ as large as possible so that $T$ contains a 
	subtournament $T'$ satisfying $|T'| \geq |T|/8^i$ and such that 
	$T'$ is $(2^i\epsilon )$-far from being transitive. Let $|T'| = t \geq n/8^i$ 
	and list the vertices of $T'$ in an optimal ordering $v_1,\ldots ,v_t$. 
	Letting $B$ denote the backwards edges of $T'$ and $|B| = \alpha t^2$, 
	we have $\alpha \geq 2^i\epsilon $. In particular, since $\alpha \leq 1$ 
	we have $1/2^i \geq \epsilon$. Now by the choice of $i$, the conclusion of  
	Lemma \ref{many long edges or density boost} part 2 fails for $T'$. Lemma 
	\ref{many long edges or density boost} therefore guarantees that the subset 
	$B'$ of $B$ consisting of edges of length at least $t/16$ satisfies 
	$|B'| \geq |B|/4 = \alpha t^2/4$. 
	
	We first consider the case when $\alpha > 2^{-16}$. Here we apply Theorem 
	\ref{F_k subtournament} to $T'$ taking advantage of the fact that $\alpha $ 
	is quite large. Indeed, as $T'$ is $\alpha 
	$-far from being transitive, by Theorem \ref{F_k subtournament} we find 
	that $T'$ contains $D_k$ as a subtournament, provided $t 
	\geq \alpha ^{-ck/\alpha ^2}$. This holds as
	\begin{align*}
		t \geq n / 8^i \geq n\epsilon ^{3}
		  \geq \epsilon ^{-Ck +3}
		  \geq \epsilon ^{-Ck/2}
		  \geq \alpha ^{-Ck/2}
		  \geq
		  \alpha ^{-2^{32}ck} \geq \alpha ^{-ck/\alpha ^2}.
	 \end{align*}
	Here we used that $1/2^i \geq \epsilon $, that $C\geq 6$ and $k\geq 1$ and 
	that $\alpha \geq \epsilon $.
	
	Now we deal with the case when $\alpha \leq 2^{-16}$. We can apply Lemma 
	\ref{many disjoint triangles} to $T'$ taking $B$ and $B'$ as given above, 
	to find a subset $B'' \subset B'$, satisfying $|B''| \geq |B'|/2 \geq 
	(\alpha /8) t^2$ with the 
	property that each edge of $B''$ lies in at least $t/64$ directed triangles 
	in $T'$. We now apply Lemma 
	\ref{many imbalanced triangles} to $T'$ taking $E = B''$, 
	$\beta = \alpha /8 $ and 
	$\gamma = 1/64$. This shows that $T'$ contains a copy of $D_k$, provided that 
	$|T'| = t \geq \beta ^{-100k/\gamma } = \beta ^{-6400 k}$. To see that this 
	holds, 
	first note that $t \geq n/8^i \geq 
	\epsilon ^{-Ck}/8^i \geq \epsilon ^{-Ck +3} \geq \epsilon ^{-Ck/2} $ 
	as $C\geq 6$. 
	Using $\beta \geq 2^i\epsilon /8 \geq \epsilon /8 
	\geq \epsilon ^4$ (since $1/2 \geq \epsilon $) gives 
	$t \geq \epsilon ^{-Ck/2} \geq \beta ^{-Ck/8} 
	\geq \beta ^{-2^{30}ck} \geq \beta ^{-6400k}$, as required.
	\end{proof}

\noindent \textbf{Acknowledgements:} I would like to thank Asaf Shapira for sharing a copy of his preprint with Raphael Yuster \cite{SY} and the referee for many useful comments.

\end{document}